\documentclass[review]{elsarticle}
\usepackage{lineno}
\usepackage[latin1]{inputenc}
\usepackage{times}
\usepackage{amssymb,mathrsfs, amsmath}
\usepackage{amsmath,amsthm}
\usepackage{amssymb}
\usepackage{latexsym}
\usepackage{amsfonts}
\usepackage{epsfig}
\usepackage{hyperref}
\usepackage{graphicx}
\usepackage{graphics}
\usepackage[all]{xy}
\usepackage{tikz-cd}
\usepackage[mathscr]{euscript}
\usepackage{mathrsfs}
\usepackage{stmaryrd}

\newtheorem{thm}{Theorem}[section]

\newtheorem{defn}{Definition}[section]
\newtheorem{prop}{Proposition}[section]

\newtheorem{lem}{Lemma}[section]

\newtheorem{rem}{Remark}[section]

\newtheorem{exmpl}{Example}[section]

\journal{... }

\DeclareMathOperator{\Hom}{Hom}
\newcommand{\lb }{\llbracket}
\newcommand{\rb}{\rrbracket}
\DeclareMathOperator{\g}{\mathfrak{g}}
\DeclareMathOperator{\h}{\mathfrak{h}}

\begin{document}
\begin{frontmatter}

\title{Characterization and Cohomology of  Crossed  Homomorphisms on  Lie Superalgebras}


%
\author{RB Yadav \fnref{myfootnote}\corref{mycorrespondingauthor}}
\cortext[mycorrespondingauthor]{Corresponding author}
\ead{rbyadav01@cus.ac.in, rbyadav15@gmail.com}
\author{Arpan Sharma\corref{mycoauthor}}
\ead{arpansachin@gmail.com}
\address{Sikkim University, Gangtok, Sikkim, 737102, \textsc{India}}

\begin{abstract}
In this article,  we give a characterisation of crossed homomorphisms on Lie superalgebras as a Maurer-Cartan element of a graded Lie algebra. Using this characterisation we study cohomology of these crossed homomorphisms. As an application of this cohomology we study formal deformation of crossed homomorphisms. We show that linear deformations of these homomorphisms are characterised by one cocycles.

\end{abstract}

\begin{keyword}
\texttt{  Cohomology, Crossed homomorphism, Deformation, Maurer-Cartan element, Lie superalgebra}
\MSC[2020]  	16S80, 	16E40,  17B56, 17B70,  17B10, 17B55
\end{keyword}
\end{frontmatter}


\section{Introduction}
Graded Lie algebras have been of significant interest in physics, particularly in the study of supersymmetries, which relate particles with different statistics. In mathematics, they have been explored in the context of deformation theory \cite{MR0438925}. The classification of Lie superalgebras was established by Kac \cite{MR486011}, while Leits \cite{MR0422372} introduced a cohomology theory for them. In physics, Lie superalgebras are also referred to as $ \mathbb{Z}_2$-graded Lie algebras.

Algebraic deformation theory was first introduced by Gerstenhaber for rings and algebras \cite{MR171807},\cite{MR0207793},\cite{MR240167}, \cite{MR389978}, \cite{MR704600}. The deformation theory of Lie superalgebras was later studied by Binegar \cite{MR871615}. The Maurer-Cartan characterization of Lie algebra structures was provided by Nijenhuis and Richardson \cite{MR0214636}, while Gerstenhaber established a similar characterization for associative algebras \cite{MR1028197}. A corresponding characterization for Lie superalgebra structures was given in \cite{MR1028197}. 

Crossed homomorphisms on groups appeared in Whitehead's early work \cite{Whitehead} and were later utilized in the study of non-abelian Galois cohomology \cite{Serre}. The concept was extended to Lie algebras in \cite{Lue} in the context of non-abelian extensions.
Crossed homomorphisms are also called relative difference operators.
More recently, in \cite{Yunhe1},\cite{Apurba},\cite{Wang}, cohomology and deformation of crossed homomorphisms have been studied.

Since cohomology of crossed homomorphisms of Lie superalgebras have not been studied, we find it interesting to study characterisation, cohomology and deformation of these homomorphisms.

Organisation of the paper is as follows: In Section \ref{RAS1}, we recall some basic definitions and results.  In Section \ref{RAS2}, we introduce a $\mathbb{Z}\times \mathbb{Z}_2$-graded Lie algebra whose Maurer-Cartan  elements characterise LieSupAct triple structures. With the help of this Lie algebra and corresponding Maurer-Cartan element $\pi+\rho+\mu$ we introduce a cochain complex whose cohomology we define as cohomology of Lie SupAct triple $(\g,\h,\rho)$. In Section \ref{RAS3}, we study formal one parameter deformation of LieSupAct triples. We show that infinitesimal of a deformation is a $1$-cocycle. We  show that linear deformations are characterised by $1$-cocycles. In Section \ref{RAS4}, we introduce  crossed homomorphisms on Lie superalgebras. We discuss category of all  crossed homomorphisms $D:\g\to\h$, for any fixed LieSupAct triple $(\g,\h,\rho)$.  In Section \ref{RAS5}, we give cohomology and characterisation of  crossed homomorphisms. In Section \ref{RAS6}, we discuss formal one parameter deformation of  crossed homomorphisms on Lie superalgebras. We show that infinitesimal of a formal  deformation  is cocycle. Also, we  show that linear deformations of these operators are characterised by $1$-cocycles.
\section{Preliminaries}\label{RAS1}

  Let $V=V_0\oplus V_1$ and $W=W_0\oplus W_1$ be $\mathbb{Z}_2$-graded vector spaces over a field $K$. A linear map $f:V \to W$ is said to  be homogeneous of degree $\alpha$ if $f(V_\beta)\subset W_{\alpha+\beta}$, for all  $\beta\in \mathbb{Z}_2= \{0,1\}$. We write $(-1)^{\deg(f)}=(-1)^f$. Elements of $ V_\beta$ are called homogeneous of degree $\beta.$  A \textbf{superalgebra} is a $\mathbb{Z}_2$-graded vector space $A=A_0\oplus A_1$ together with a bilinear map $m:A\times A\to A$ such that $m(a,b)\in A_{\alpha+\beta},$ for all  $a\in A_{\alpha}$, $b\in A_{\beta}$. 
\begin{defn}
A  \textbf{ Lie superalgebra } is a a $\mathbb{Z}_2$-graded vector space  $L=L_0\oplus L_1$ over a field $K$ equipped with an operation $[-,-]:L\times L\to L$ satisfying the following conditions:
\begin{enumerate}
\item  $(L, [-,-])$ is  superalgebra,
  \item $[a,b]=-(-1)^{\alpha\beta}[b,a]$, 
  \item  $[a,[b,c]]=[[a,b],c]+(-1)^{\alpha\beta}[b,[a,c]]$, \hspace{3cm}(Jacobi identity)
\end{enumerate}
for all $a\in L_{\alpha}$and $b\in L_{\beta}$.
  Let $L_1$ and $L_2$ be two Lie superalgebras. A homomorphism $f:L_1\to L_2$ is  a $K$-linear map such that $f([a,b])=[f(a),f(b)].$  Given a Lie superalgebra $L$,     the vector subspace of $L$ spanned  by the set $\{[x,y]: x,y\in L\}$ is denoted by $[L,L]$. A Lie superalgebra $L$ is called abelian if $[L,L]=0.$
\end{defn}
\begin{exmpl}\label{rbLSe1}
 Let $V=V_{\bar{0}}\oplus V_{\bar{1}}$ be a $\mathbb{Z}_2$-graded vector  space, $dim V_{\bar{0}} = m$, $dim V_{\bar{1}} = n$. Consider the  associative algebra $End V$ of all endomorphisms of $V$. 
Define
\begin{equation}
\tag{1}
 End_i \; V = \{a \in End \;V\;|\; a V_s\subseteq V_{i+s}\} \; ,\; i,s\in \mathbb{Z}_2
 \end{equation}
 One can easily verify that  $End\; V=End_{\bar{0} }\; V\oplus End_{\bar{1}} \; V  $.
The bracket $[a,b] = ab - (-1)^{\bar{a} \bar{b}}ba$ makes $End V$ into a Lie superalgebra, denoted by $\ell(V)$ or $\ell(m,n)$. In some (homogeneous) basis of $V$, $\ell(m,n)$ consists of  block matrices of the form $\big(\begin{smallmatrix}\alpha & \beta \\ \gamma & \delta \end{smallmatrix}\big)$, where $\alpha, \beta, \gamma,  \delta$ are matrices of order $m\times m$, $m\times n$, $n\times m$ and $n\times n,$ respectively.
\end{exmpl}
Let $V=V_0\oplus V_1$ be a   $\mathbb{Z}_2$-graded vector space.
 Consider the permutation group $S_n.$ For any $X=(X_1,\ldots, X_n)$ with $X_i\in V_{x_i}$ and $\sigma\in S_n$, define 
$$ K(\sigma, X)= card\{ (i,j): i<j,\; X_{\sigma(i)}\in V_1, X_{\sigma(j)}\in V_1,  \; \sigma(j) < \sigma(i)\}, $$
$$\epsilon(\sigma, X)=\epsilon(\sigma)(-1)^{K(\sigma,X)},$$
where $ card\; A $ denotes cardinality of a set $A, $ $\epsilon(\sigma)$ is the signature of $\sigma.$ Also,  define  $\sigma.X=(X_{\sigma^{-1}( 1)},\ldots, X_{\sigma^{-1}( n)}).$
We have following Lemma \cite{MR1028197},
\begin{lem}\label{SFL1}
    \begin{enumerate}
        \item $K(\sigma \sigma',X)=K(\sigma,X)+K(\sigma',\sigma^{-1}X)\;\;\; (mod\; 2).$ 
        \item $\epsilon(\sigma \sigma', X)=\epsilon(\sigma, X)\epsilon(\sigma', 
\sigma^{-1} X)$.
    \end{enumerate}
\end{lem}
Let $V$ and $W$ be  $\mathbb{Z}_2$-graded  vector spaces.  For each  $n\in \mathbb{N},$ define   $\mathcal{F}_{n,\alpha}(V,W)$ as  the vector space of all homogeneous $n$-linear mappings  $f:V \underset{n\; times}{\underbrace{\times\cdots\times}}V\to W$ of degree $\alpha.$ Denote the set of all $n$-linear mappings  $f:V \underset{n\; times}{\underbrace{\times\cdots\times}}V\to W$ by  $\mathcal{F}_{n}(V,W)$.  Clearly, we have   $\mathcal{F}_{n}(V,W)=\mathcal{F}_{n,0}(V,W)\oplus \mathcal{F}_{n, 1}(V,W)$. Write  $\mathcal{F}_{0}(V,W)=W$ and  $\mathcal{F}_{-n}(V,W)=0$,  $\forall n\in \mathbb{N}.$  Take $\mathcal{F}(V,W)=\bigoplus_{n\in \mathbb{Z}}\mathcal{F}_{n}(V,W)$. 

For $F\in \mathcal{F}_n(V,W)$, $X\in V^n,$ $\sigma\in S_n,$ define 
$$(\sigma.F)(X)=\epsilon(\sigma, X)F(\sigma^{-1}X).$$ 

By  using Lemma \ref{SFL1},  one concludes that this defines an action of   $S_n$    on the  $\mathbb{Z}_2$-graded vector space  $\mathcal{F}_n(V,W)$.

Define  $\mathcal{\tilde{F}}_n(V,W)$  for  $n\in \mathbb{Z}$ as follows:\\
 Set  $\mathcal{\tilde{F}}_n(V,W)=\{F\in \mathcal{F}_{n}(V,W) : \sigma.F=F,\;\forall \;\sigma\in S_{n} \}$, for $n\ge 1$ and 
$$\mathcal{\tilde{F}}_n(V,W)=  \begin{cases}
    W& \textit{if}\; n=0\\
    0& \textit{if}\;n\le  -1
\end{cases}.$$
    We know that $\mathcal{\tilde{F}}_n(V,W)\cong  Hom(\wedge^nV, W)$. Write  $\mathcal{E}_n= Hom(\wedge^{n+1}V, V)$ and  $\mathcal{E}=\bigoplus_{n\in \mathbb{Z}}\mathcal{E}_n$. If we denote by $\mathcal{E}_{n,\alpha}$ the homogeneous elements of degree $\alpha\in \mathbb{Z}_2$, then $\mathcal{E}_n=\mathcal{E}_{n,0}\oplus \mathcal{E}_{n,1}$. Define a product $\circ$ on $\mathcal{E}$ as follows:
For $F\in \mathcal{E}_{n,f}$,  $F'\in \mathcal{E}_{n',f'}$  define 
\begin{equation}\label{RA20}
    F\circ F'=\sum_{\sigma\in S_{(n,n'+1)}}\sigma.(F*F'),
\end{equation}
where $$F*F'(X_1, \ldots, X_{n+n'+1})=(-1)^{f'(x_1+\cdots+x_n)}F(X_1,\ldots,X_n, F'(X_{n+1},\ldots,X_{n+n'+1})),$$ for $X_i\in V_{x_i}$, and $S_{(n,n'+1)}$ consists of permutations $\sigma\in S_{n+n'+1}$ such that  $\sigma(1)<\cdots<\sigma(n)$, $\sigma(n+1)<\cdots<\sigma(n+n'+1).$
Clearly, $F\circ F'\in \mathcal{E}_{(n+n',f+f')}$. We have following  Lemma \cite{MR1028197}.
\begin{lem}\label{PLA1}
 For $F\in \mathcal{E}_{n,f}$,  $F'\in \mathcal{E}_{n',f'}$,   $F''\in \mathcal{E}_{n'',f''}$  
 $$(F\circ F')\circ F'' -F\circ (F'\circ F'')= (-1)^{n'n''+f'f''}\{(F\circ F'')\circ F'-F\circ ( F''\circ F')\}.$$
\end{lem}

Using Lemma \ref{PLA1}, we have following theorem \cite{MR195995}.
\begin{thm}\label{MLT3}
$\mathcal{E}$ is a $\mathbb{Z}\times \mathbb{Z}_2$-graded Lie algebra
  with the bracket  $[\;,\;]$ defined  by 
  \begin{equation}\label{GNR}
      [F,F']=F\circ F'-(-1)^{nn'+ff'}F'\circ F
  \end{equation}
  for $F\in \mathcal{E}_{n,f}$,  $F'\in \mathcal{E}_{n',f'}$
\end{thm}
For any $\mathbb{Z}_2$-graded vector space $V=V_0\oplus V_1$, let $\mathrm{End}(V)$ denote the vector space of endomorphisms of $V$. For $\alpha\in \mathbb{Z}_2,$ if we define  $\mathrm{End}_\alpha (V)={f\in \mathrm{End}(V): f(V_\beta})\subset V_{\alpha+\beta}$, then we can write   $\mathrm{End}(V) = \mathrm{End}_0 (V) \oplus \mathrm{End}_1( V) $. The bracket $[a,b] = ab - (-1)^{|a||b|} ba$ makes $\mathrm{End} (V)$ into a Lie superalgebra.

\begin{defn}
Let  $s \in \mathbb{Z}_2$ . A derivation of degree $s$ of a superalgebra $A$ is an endomorphism $D \in \mathrm{End}_s(A)$ with the property $$D(ab)=D(a)b + (-1)^{s \cdot |a|}aD(b)$$
where $\mathrm{End}_sA$ is the space of all endomorphisms of degree $s$ on $A$, $|a|$ denotes the degree of $a$ and $\mathrm{Der}_s(A)$ is the space of all derivations of degree $s$. 
\end{defn}
\begin{rem}
    \begin{enumerate}
\item If we write  $\mathrm{Der}(A)=\mathrm{Der}_0(A) \oplus \mathrm{Der}_1(A)$, then space $\mathrm{Der}(A)$ is vector subspace of $ \mathrm{End}(A)$.
\item  $\mathrm{Der}(A)$ is closed under the  bracket $[a,b] = ab - (-1)^{|a||b|} ba$. Hence it is a subalgebra of $\mathrm{End}(A)$ and is called the \textbf{superalgebra of derivations} of $A$. Every element of $\mathrm{Der}(A)$ is called \textbf{derivation} of $A$.
    \end{enumerate}
\end{rem} 

\section{Characterization and Cohomologies of LieSupAct Triples}\label{RAS2}

\begin{defn}
    Let $\mathfrak{g}, [-,-]_{\mathfrak{g}}$ and $\mathfrak{h}, [-,-]_{\mathfrak{h}}$ be two Lie superalgebras. A Lie superalgebra morphism $\rho: \mathfrak{g}\to \mathrm{Der}(\mathfrak{h})$ of degree $0$ is called an action of  $\mathfrak{g}$  on $\mathfrak{h}$. We call the  triple $(\mathfrak{g}, \mathfrak{h}, \rho)$ a \textbf{LieSupAct triple}. 
\end{defn} 
Let  $S_{(l,m,n,p)}$ consist of permutations $\sigma\in S_{l+m+n+p}$ such that  $$\sigma(1)<\cdots<\sigma(l),\;\; \sigma(l+1)<\cdots<\sigma(l+m),$$  $$\sigma(l+m+1)<\cdots<\sigma(l+m+n), \;\;\sigma(l+m+n+1)<\cdots<\sigma(l+m+n+p).$$
Clearly, $S_{(l+m,n+p)}\subset S_{(l,m,n,p)}$.
Let $\mathfrak{g}=\mathfrak{g}_0\oplus \mathfrak{g}_1$ and $\mathfrak{h}=\mathfrak{h}_0\oplus \mathfrak{h}_1$ be two $\mathbb{Z}_2$
graded vector spaces. For any linear maps $\alpha: \wedge ^l \mathfrak{g_0}\otimes \wedge^m\mathfrak{g_1}\otimes \wedge^n\mathfrak{h_0}\otimes \wedge^p\mathfrak{h_1}\to \mathfrak{g} $ and $\beta: \wedge ^l \mathfrak{g_0}\otimes \wedge^m\mathfrak{g_1}\otimes \wedge^n\mathfrak{h_0}\otimes \wedge^p\mathfrak{h_1}\to \mathfrak{h} $, define $\hat{\alpha}, \hat{\beta}\in Hom( \wedge^{l+m+n+p}(\mathfrak{g}\oplus \mathfrak{h}), \mathfrak{g}\oplus \mathfrak{h})$ by
\begin{align}\label{ext1}
   &\hat{\alpha}( x_0^1+x_1^1+y_0^1+y_1^1,\ldots, x_0^{l+m+n+p}+x_1^{l+m+n+p}+y_0^{l+m+n+p}+y_1^{l+m+n+p})\nonumber\\
  =&\sum_{\sigma\in S_{(l,m,n,p)}} \big((\sigma.\alpha) Z, 0\big)  
\end{align}
\begin{align}\label{ext2}
   &\hat{\beta}( x_0^1+x_1^1+y_0^1+y_1^1,\ldots, x_0^{l+m+n+p}+x_1^{l+m+n+p}+y_0^{l+m+n+p}+y_1^{l+m+n+p})\nonumber\\
  =&\sum_{\sigma\in S_{(l,m,n,p)}} \big(0,(\sigma.\beta) Z\big) 
\end{align}
Here 
    $Z=(x_0^1,\ldots, x_0^l,x_1^{l+1}\ldots,x_1^{l+m}, y_0^{l+m+1},\ldots, y_0^{l+m+n},y_1^{l+m+n+1},\ldots, y_1^{l+m+n+p}),$ 
     \begin{scriptsize}
    \begin{align*}
&\sigma^{-1}Z \\
&=(x_0^{\sigma(1)},\ldots, x_0^{\sigma(l)},x_1^{\sigma(l+1)}\ldots,x_1^{\sigma(l+m)}, y_0^{\sigma(l+m+1)},\ldots, y_0^{\sigma(l+m+n)},y_1^{\sigma(l+m+n+1)},\ldots, y_1^{\sigma(l+m+n+p)}).
\end{align*} 
\end{scriptsize}
We write $\g^{l,m,n,p}=\wedge ^l \mathfrak{g_0}\otimes \wedge^m\mathfrak{g_1}\otimes \wedge^n\mathfrak{h_0}\otimes \wedge^p\mathfrak{h_1}$. Then by Kunneth formula  we have $\wedge^{s}(\g \oplus \h)=\oplus_{s=l+m+n+p}\g^{l,m,n,p}.$  From  \ref{ext1} and \ref{ext2}, we get an isomorphism \begin{scriptsize}  $\hat{.}:(\oplus_{l+m+n+p=s} Hom(\g^{l,m,n,p}, \g))\oplus (\oplus_{l+m+n+p=s} Hom(\g^{l,m,n,p}, \h))\to Hom(\wedge^s(\g \oplus \h), \g \oplus \h).$\end{scriptsize}

\begin{prop}\label{subalg1}
    The $\mathbb{Z}\times \mathbb{Z}_2$-graded vector space  $$\mathfrak{F}= \oplus_{s=0}^\infty (Hom(\wedge^{s+1} \g, \g)\oplus(\oplus_{\substack{i+j=s+1\\j\ge 1}} Hom(\wedge^i \g \otimes \wedge^j \h, \h)))$$
    is a graded Lie subalgebra of $\mathbb{Z}\times \mathbb{Z}_2$-graded Lie algebra of $(\mathscr{G},[.;.])$, where \\$\mathscr{G}=\oplus_{s=0}^\infty Hom(\wedge^{s+1}(\g \oplus \h), \g \oplus \h )$ and    $[.;.]$ is the bracket defined in \ref{GNR}.
\end{prop}

\begin{proof}
\textbf{Case-I:}\\
Let $F\in Hom(\wedge^{l}\g_0 \otimes \wedge^m \g_1, \g)$, $G\in Hom(\wedge^{p}\g_0 \otimes \wedge^q \g_1, \g)$ be homogeneous of degree $(a,f)$ and $(b,g)$ respectively. Thus $a=l+m-1$, $b=p+q-1$. Then by definition $$[F,G] =[\hat{F},\hat{G}]=\hat{F}\circ \hat{G}-(-1)^{ab+fg}\hat{G}\circ \hat{F}.$$ By \ref{RA20}, \ref{ext1} and  \ref{ext2}, it is clear that 
$[\hat{F},\hat{G}]Z =0$  if $ \;Z\in \wedge^i \g \otimes \wedge^j_{j\ge 1}\h$ and 
$[\hat{F},\hat{G}]Z \in \g$ if $Z\in \wedge^{a+b+1}\g$.\\
\textbf{Case-II:}\\
Let   $F\in Hom(\wedge^{l}\g_0\otimes \wedge^m \g_1, \g) $,  $G\in Hom(\wedge^{p}\g_0 \otimes \wedge^q \g_1 \otimes \wedge^{r}\h_0 \otimes \wedge^s \h_1, \h)$, where $r+s\ge 1$,  be  homogeneous of degree $(a,f)$ and $(b,g)$, respectively. Thus $b=p+q+r+s-1$, $a=l+m-1$.
Then by definition $$[F,G] =[\hat{F},\hat{G}]=\hat{F}\circ \hat{G}-(-1)^{ab+fg}\hat{G}\circ \hat{F}.$$ By \ref{RA20}, \ref{ext1} and  \ref{ext2}, it is clear that $\hat{F}\circ\hat{G}=0,$  $(\hat{G}\circ\hat{F}) Z=0$ if $Z\in \wedge^{i}\g$  and  $\hat{G}\circ\hat{F}\in \oplus_{\substack{i+j=l+m+p+q+r+s-1\\j\ge 1}} Hom(\wedge^i \g \otimes \wedge^j \h, \h).$  Hence $[F,G]\in \mathfrak{F}$. \\
\textbf{Case-III:}\\
Let $F\in Hom(\wedge^{p'}\g_0\otimes \wedge^{q'} \g_1\otimes \wedge^{r'}\h_0\otimes \wedge^{s'} \h_1, \h)$, $G\in Hom(\wedge^{p}\g_0\otimes \wedge^q \g_1\otimes \wedge^{r}\h_0\otimes \wedge^s \h_1, \h)$, be  homogeneous of degree $(a,f)$ and $(b,g)$, respectively, where $r+s\ge 1$, $r'+s'\ge 1$. Thus $b=p+q+r+s-1$, $a=p'+q'+r'+s'-1$. Then by definition $$[F,G] =[\hat{F},\hat{G}]=\hat{F}\circ \hat{G}-(-1)^{ab+fg}\hat{G}\circ \hat{F}.$$ By \ref{RA20}, \ref{ext1} and  \ref{ext2}, it is clear that $[\hat{F},\hat{G}]Z=0$ if $Z\in \wedge^{i}\g$ and  $$[\hat{F},\hat{G}]\in \oplus_{\substack{i+j=p+q+r+s+p'+q'+r'+s'-1\\j\ge 1}} Hom(\wedge^i \g \otimes \wedge^j \h, \h).$$  Hence $[F,G]\in \mathfrak{F}$. 
This completes the proof.
\end{proof}
\begin{defn}\label{RAdf30} We define  a graded derivation of degree $(\alpha,\beta)$ of $\mathbb{Z}\times \mathbb{Z}_2$-graded Lie algebra $\mathcal{E}$ as is a linear map $\partial :\mathcal{E}\to \mathcal{E}$ of degree $(\alpha,\beta)$ such that $\partial ([x,y])=[\partial x,y]+(-1)^{\alpha l +\beta m}[x,\partial y]$, for every homogeneous  $x$ of degree $(l,m)$.
A  differential $\mathbb{Z}\times \mathbb{Z}_2$-graded Lie algebra   $(\mathcal{E}, [.,.],\partial)$ consists of  a $\mathbb{Z}\times \mathbb{Z}_2$-graded  Lie algebra   $(\mathcal{E}, [.,.])$ and a graded derivation $\partial :\mathcal{E}\to \mathcal{E}$ of degree $ (1,0)$ such that $\partial^2=0$.   Maurer-Cartan elements of $(\mathcal{E}, [.,.],\partial)$ is defined as an element $\omega$ of $\mathcal{E}$  of degree $(1,0)$ such that $\partial\omega+\frac{1}{2}[\omega,\omega]=0.$ Every $\mathbb{Z}\times \mathbb{Z}_2$-graded Lie algebra $(\mathcal{E}, [.,.])$ can be seen as   differential graded Lie algebra with $\partial=0$.
  Hence Maurer-Cartan elements of the $\mathbb{Z}\times \mathbb{Z}_2$-graded  Lie algebra   $\mathcal{E}$ is an element $\omega$ of $\mathcal{E}$  of degree $(1,0)$ such that $[\omega,\omega]=0.$
\end{defn}

\begin{thm}\label{RA1}
Let $\g$ and $\h$ be $\mathbb{Z}_2$-graded vector space. Then Maurer-Cartan elements of the $\mathbb{Z}\times \mathbb{Z}_2$-graded  Lie algebra $(\mathfrak{F}, [.;.])$ defined in \ref{subalg1} characterize  LieSupAct triple structures.
\end{thm}
\begin{proof}
Let $\omega$ be   Maurer- Cartan element of $(\mathfrak{F}, [.;.])$. Hence $w=\pi +\rho+\mu\in  Hom(\wedge^{2} \g, \g) \oplus Hom( \g \wedge \h, \h)\oplus Hom( \wedge^2 \h, \h)))$  such that  $$[\pi +\rho+\mu,\pi +\rho+\mu]=0.$$  Note that $[\pi +\rho+\mu,\pi +\rho+\mu]\in Hom(\wedge^{3} \g, \g)\oplus Hom(\wedge^2 \g \otimes \wedge^1 \h, \h)\oplus Hom(\wedge^1 \g \otimes \wedge^2 \h, \h) \oplus Hom(\wedge^3 \h, \h)$.  Now,
 \begin{align}\label{RA30}
  0 &=  [\pi +\rho+\mu,\pi +\rho+\mu]\nonumber\\
     &= [\pi ,\pi ]+ [\pi ,\rho]+[\pi ,\mu]+[\rho ,\pi ]+ [\rho ,\rho]+[\rho ,\mu]+[\mu ,\pi ]+ [\mu ,\rho]+[\mu ,\mu]\nonumber\\
     &= [\pi ,\pi ]+ 2\rho\circ \pi  + [\rho ,\rho]+ 2[\rho ,\mu]+[\mu ,\mu]
 \end{align} 
 Clearly,  $[\pi ,\pi ]\in Hom(\wedge^{3} \g, \g) $,  $2\rho\circ \pi\  + [\rho ,\rho]\in Hom(\wedge^2 \g \otimes \wedge^1 \h, \h) $,  $2[\rho ,\mu] \in Hom(\wedge^1 \g \otimes \wedge^2 \h, \h)  $  and  $[\mu ,\mu] \in Hom(\wedge^3 \h, \h)$.
 Hence, by using \ref{RA30}, $ [\pi ,\pi ]=0$,  $2\rho\circ \pi  + [\rho ,\rho]=0$, $2[\rho ,\mu]=0$ and  $[\mu ,\mu]=0$.   $ [\pi ,\pi ]=0$ if and only if $(\g,\pi)$ is a Lie superalgebra.  $[\mu ,\mu ]=0$ if and only if $(\h,\mu)$ is a Lie superalgebra. 
 Since $2[\rho ,\mu]=0$,  for $(x_1, x_2, x_3)\in \wedge^1 \g\otimes \wedge^2 \h$, we have 
 \begin{align}\label{RA35}
  0&=[\rho ,\mu](x_1, x_2, x_3)\nonumber\\
  &=(\rho\circ \mu +\mu\circ \rho)(x_1, x_2, x_3)\nonumber\\
  &= \rho (x_1, \mu(x_2, x_3))-(-1)^{x_1 x_2}\rho (x_2, \mu(x_1, x_3))+ (-1)^{x_1x_3+x_2x_3}\rho (x_3, \mu(x_1, x_2))\nonumber\\
  &+ \mu (x_1, \rho(x_2, x_3))-(-1)^{x_1 x_2} \mu (x_2, \rho(x_1, x_3))+(-1)^{x_1x_3+x_2x_3}\mu (x_3, \rho(x_1, x_2))\nonumber\\
   &= \rho (x_1, \mu(x_2, x_3))-(-1)^{x_1 x_2} \mu (x_2, \rho(x_1, x_3))+(-1)^{x_1x_3+x_2x_3}\mu (x_3, \rho(x_1, x_2))
 \end{align}
 \ref{RA35} is equivalent to following 
 \begin{equation}\label{RA36}
      \rho (x_1, \mu(x_2, x_3))=\mu ( \rho(x_1, x_2), x_3)+(-1)^{x_1 x_2} \mu (x_2,\rho(x_1, x_3))
 \end{equation}
 \ref{RA36} defines a linear map $\hat{\rho}:\g \to \mathrm{Der}(\h)$  given by $\hat{\rho}(x_1)(x_2)=\rho(x_1, x_2)$.
  \begin{align}\label{RA37}
  0&=(2\rho\circ \pi+[\rho ,\rho])(x_1, x_2, x_3)\nonumber\\
  &=2(\rho\circ \pi +\rho\circ \rho)(x_1, x_2, x_3)\nonumber\\
  &=2\{ \rho (x_1, \pi(x_2, x_3))-(-1)^{x_1 x_2}\rho (x_2, \pi(x_1, x_3))+ (-1)^{x_1x_3+x_2x_3}\rho (x_3, \pi(x_1, x_2))\nonumber\\
  &+ \rho(x_1, \rho(x_2, x_3))-(-1)^{x_1 x_2} \rho (x_2, \rho(x_1, x_3))+(-1)^{x_1x_3+x_2x_3}\rho (x_3, \rho(x_1, x_2))\}\nonumber \\
   &= \rho (x_1, \rho(x_2, x_3))-(-1)^{x_1 x_2} \rho (x_2, \rho(x_1, x_3))+(-1)^{x_1x_3+x_2x_3}\rho (x_3, \pi(x_1, x_2))\nonumber\\
    &= \rho (x_1, \rho(x_2, x_3))-(-1)^{x_1 x_2} \rho (x_2, \rho(x_1, x_3))-\rho (\pi(x_1, x_2), x_3)
 \end{align}
From \ref{RA37}, we conclude that   $\hat{\rho}:\g \to \mathrm{Der}(\h)$  given by $\hat{\rho}(x_1)(x_2)=\rho(x_1, x_2)$ is  a Lie superalgebra morphism. Given any  LieSupAct triple $(\g,\h,\rho)$, we can repeat above arguements in opposite direction to get a corresponding  Maurer-Cartan elements $\pi +\rho+\mu$.  Hence we conclude that  Maurer-Cartan elements $\pi +\rho+\mu$ characterize LieSupAct triple structiures.
\end{proof}

Let $(\g, \pi)$, $(\h,\mu)$ be Lie superalgebras and $\rho:\g \to \mathrm{Der}(\h)$ be an action of $\g$ on $\h.$ Then by Theorem \ref{RA1}, $\Pi=\pi+\rho+\mu$ is a Maurer-Cartan element of the graded Lie algebra $(\mathfrak{F}, [.;.])$. Define a linear map $\partial_\Pi: \mathfrak{F}\to \mathfrak{F}$ by $\partial_\Pi (a)=[\Pi, a]$.

\begin{thm}\label{RA2}
 Let $(\g, \pi)$, $(\h,\mu)$ be Lie superalgebras and $(\g,\h,\rho)$ be a LieSupAct triple. Then  $(\mathfrak{F}, [.;.], \partial_\Pi)$ is a differential graded Lie algebra.
\end{thm}
\begin{proof}
    By using graded Jacobi identity in $(\mathfrak{F}, [.;.])$, we conclude that $\partial_\Pi$ is a graded derivation of degree $1$. Also, the anticommutativity and graded Jacobi identity conditions  imply that $\partial_\Pi^2=0.$
\end{proof}
Next  we shall define cohomology of a LieSupAct triple $(\g,\h,\rho)$ using Theorem \ref{RA2}.  For $n\ge 1$, define $C^n(\g,\h,\rho)=Hom(\wedge^n \g, \g) \oplus \left(\oplus_{i=0}^{n-1}Hom(\wedge^i \g \otimes \wedge^{n-i}\h, \h)\right).$
Define $\partial^n:C^n(\g,\h,\rho)\to C^{n+1}(\g,\h,\rho)$ by $\partial^n(a)=\partial_\Pi(a)=[\Pi,a].$ Then  $\oplus_n \partial^n=\partial_\Pi.$ Write $C^*(\g,\h,\rho)=\oplus_{n=1}^\infty C^n(\g,\h,\rho).$  By Theorem \ref{RA2},
 $(C^*(\g,\h,\rho), \partial_\Pi)$ is a cochain complex. 
We define the cohomology of the cochain complex  $(C^*(\g,\h,\rho), \partial_\Pi)$ to be the \textbf{cohomology of the LieSupAct triple} $(\g,\h,\rho)$. We denote the $n$th cohomology group of $(C^*(\g,\h,\rho), \partial_\Pi)$ by $H^n(\g,\h,\rho)$.

\section{Formal One Parameter Deformation of LieSupAct Triples}\label{RAS3}

 Let $(\g,\pi)$ and $(\h,\mu)$  be  any two Lie superalgebras.  Let $(\g,\h,\rho)$  be a LieSupAct triple.  We write $\g[[t]]=\{\sum_{i=0}^\infty x_it^i: x_i\in \g,\; \forall \; i\}$, $\h[[t]]=\{\sum_{i=0}^\infty y_it^i: x_i\in \h,\; \forall \; i\}$. We write   $\rho_t=\sum_{i=0}^\infty \rho_it^i$,  $\pi_t=\sum_{i=0}^\infty \pi_it^i$, $\mu_t=\sum_{i=0}^\infty \mu_it^i$, where  $\rho_i:\g\to End(\h)$, $\pi_i:\g\wedge \g\to \g$, $\mu_i:\h\wedge \h\to \h $ are linear maps of degree $0$ for every $i\ge 0$ and  $\rho_0=\rho,$  $\pi_0=\pi,$ $\mu_0=\mu.$ 
\begin{defn}\label{RA330}
We define a formal one parameter deformation of  $(\g,\h, \rho)$ to be a triple $(\pi_t,\mu_t, \rho_t)$ such that following conditions are  satisfied:
\begin{enumerate}
    \item \label{RA365} $(\g[[t]], \pi_t)$  and $(\h[[t]], \mu_t)$  are Lie superalgebras.
 \item \label{RA366}  $\rho_t:\g[[t]]\to Der(\h[[t]])$ is a Lie superalgebra homomorphism.  
\end{enumerate}
 If $\pi_1+ \mu_1+\rho_1\ne 0,$ then it is called infinitesimal of the deformation $(\pi_t,\mu_t, \rho_t)$.
\end{defn}
Condition \ref{RA365} and  \ref{RA366}  are  equivalent to following conditions:
    \begin{equation}\label{RA337}
        [\pi_t,\pi_t]=0
    \end{equation}
    \begin{equation}\label{RA338}
        [\mu_t,\mu_t]=0
    \end{equation}
   \begin{equation}\label{RA339}
       \rho_t\pi_t(u,v)x= \rho_t(u)\rho_t(v)x-(-1)^{\alpha\gamma}\rho_t(v)\rho_t(u)x
    \end{equation}
   \begin{equation}\label{RA340} 
   \rho_t(u)\mu_t(x,y)= \mu_t(\rho_t(u)x,y)+(-1)^{\alpha\beta}\mu_t(x,\rho_t(u)y),
\end{equation}
for every   $u\in \g_\alpha,$ $v\in \g_\gamma$, $ x\in \h_{\beta}$. $\forall n\ge 0$,
Equations \ref{RA337}, \ref{RA338}, \ref{RA339} \ref{RA340} are  equivalent to following:
\begin{equation}\label{RA341}
        \sum_{i+j=n}[\pi_i,\pi_j](u,v,w)=0,\forall u,v,w\in \g.
    \end{equation}
    \begin{equation}\label{RA342}
        \sum_{i+j=n}[\mu_i,\mu_j](x,y,z)=0,  \forall x,y,z\in \h.
    \end{equation}
   \begin{equation}\label{RA343}
       \sum_{i+j=n} \rho_i\pi_j(u,v)x=  \sum_{i+j=n}\{\rho_i(u)\rho_j(v)x-(-1)^{\alpha\gamma}\rho_i(v)\rho_j(u)x\}, \forall u\in \g_\alpha, v\in \g_\gamma.
    \end{equation}
\begin{equation}\label{RA344}
  \sum_{i+j=n}\rho_i(u)\mu_j(x,y)= \sum_{i+j=n}\{\mu_i(\rho_j(u)x,y)+(-1)^{\alpha\beta}\mu_i(x,\rho_j(u)y)\}, \forall u\in \g_\alpha, x\in \h_{\beta}
\end{equation}
In particular, for $n=1$ we have 
$[\pi,\pi_1]=0,$ $[\mu,\mu_1]=0,$ $[\pi,\rho_1]+[\rho,\pi_1]+[\rho,\rho_1]=0,$  and $[\rho,\mu_1]+[\mu,\rho_1]=0,$ which is equivalent to $[\pi+\rho+\mu, \pi_1+\rho_1+\mu_1]=0,$ that is $\partial_{\Pi}( \pi_1+\rho_1+\mu_1)=0.$
 Thus  $\pi_1+\rho_1+\mu_1$ is a $1$-cocycle for the cochain complex $(C^*(\g,\h,\rho),\partial_{\Pi})$.  Therefore  we have following lemma:
\begin{lem}
    Infinitesimal of a deformation of $(\g,\h,\rho)$ is a $1$-cocycle.
\end{lem}

\begin{defn}\label{RA360}
 We write   $\rho_t=\rho+\rho_1t$,  $\pi_t=\pi+ \pi_1t$, $\mu_t=\mu+ \mu_1t$, where  $\rho_1:\g\to \hom(\h,\h)$, $\pi_1:\g\wedge \g\to \g$, $\mu_1:\h\wedge \h\to \h $ are linear maps of degree $0$ for every $i\ge 0$ and  $\rho_0=\rho,$  $\pi_0=\pi,$ $\mu_0=\mu.$ 
We define a linear (or degree one or infinitesimal) one parameter deformation of  $(\g,\h, \rho)$ to be a triple $(\pi_t,\mu_t, \rho_t)$ such that following conditions are  satisfied:
\begin{enumerate}
    \item \label{RA335} $(\g[[t]]/(t^2), \pi_t)$  and $(\h[[t]]/(t^2), \mu_t)$  are Lie superalgebras.
 \item \label{RA336}  $\rho_t:\g[[t]]/(t^2)\to Der(\h[[t]]/(t^2))$ is a Lie superalgebra homomorphism.  
\end{enumerate}
\end{defn}
From Equations \ref{RA341}, \ref{RA342}, \ref{RA343} and \ref{RA344}, we have following theorem:
\begin{thm}
   $(\pi_t,\mu_t, \rho_t)$ is a linear one parameter deformation of  $(\g,\h, \rho)$ if and only if $\pi_1+\rho_1+\mu_1$ is a $1$-cocycle for the cochain complex $(C^*(\g,\h,\rho),\partial_{\Pi})$
\end{thm}

\section{Crossed homomorphisms on Lie Superalgebras}\label{RAS4}

\begin{defn}
    Let $(\g,\h,\rho)$ be a LieSupAct triple. We define a linear map $D:\g\to \h$ of degree $0$ to be  a \textbf{crossed homomorphism}  with respect to the action $\rho$ if $$D([x,y])=\rho(x)D(y)- (-1)^{\alpha\beta}\rho(y)D(x)+[D(x), D(y)],\;\;\forall x\in \g_\alpha, y\in \g_{\beta}.$$
    
  Crossed homomorphisms are also called as relative difference operators in literature. We call a crossed homomorphism from $\g$ to $\g$ with respect to the adjoint action   a \textbf{difference operator} on $\g$. 
   We note that if $\rho=0,$ then $D$ is simply a Lie superalgebra homomorphism.
 \end{defn}
Next we recall    semi-direct product $\g \ltimes_\rho \h$  which is a Lie superalgebra consisting  of $\mathbb{Z}_2$-graded vector space $\g \oplus \h = (\g_0 \oplus \h_0)\oplus (\g_1 \oplus \h_1)$  and a  bracket \cite{GARCIAMARTINEZ2015464}  given  by $$[x+u,y+v]= [x,y]+\rho(x)v-(-1)^{\alpha\beta}\rho(y)u+[u,v],$$   for all homogeneous  $y\in \g_\beta, u\in \h_\alpha,$ and  $ x\in \g, v\in \h.$

Next theorem gives a characterization of crossed homomorphism in terms of graph of $D.$
\begin{thm}
    Let $(\g,\h,\rho)$ be a LieSupAct triple. Then a linear map $D:\g\to \h$ is a crossed homomorphism if and only if the graph $G(D)=\{(x,D(x)): x\in \g\}$ of $D$ is a Lie subalgebra of the semi-direct product Lie superalgebra $\g \ltimes_\rho \h.$ 
\end{thm}
\begin{proof}
By definition  $[x+D(x),y+D(y)]=[x,y]+\rho(x) D(y)-(-1)^{\alpha\beta}\rho(y)D(x)+[D(x),D(y)],$ for every homogeneous $x, y\in \g$  of degree $\alpha$ and $\beta,$ respectively.  Clearly, $D$ is a crossed homomorphism if and only if 
\begin{equation}\label{RA40}
   [x+D(x),y+D(y)]=[x,y]+D([x,y]),\;\; \forall x,y\in \g. 
\end{equation}
\ref{RA40} holds if and only if $G(D)$ is a Lie subalgebra of the semi-direct product Lie superalgebra $\g \ltimes_\rho \h.$
 Hence we conclude the result.
\end{proof}

\subsection{Category of crossed homomorphisms}
\begin{defn}\label{RAD01}  Let $(\g , \h,\rho)$ be a LieSupAct triple.
   Let $D$ and $D'$ be two crossed homomorphisms from $\g$ to $\h$. A morphism from $D$ to $D'$ consists of  Lie superalgebra homomorphisms $\phi_1:\g \to \g$ and $\phi_2:\h \to \h$ of degree $0$ such that 
    \begin{enumerate}
        \item $D'\circ\phi_1= \phi_2\circ D$
        \item $\phi_2\rho(x)(u)=\rho(\phi_1(x))(\phi_2(u)),$ $\forall x\in \g, u\in \h.$
\end{enumerate}
If $\phi_1$ and $\phi_2$ are bijective, then we say that  $D$ and $D'$ are isomorphic.
\end{defn} 
We denote the set of morphisms from $D $ to $D'$ by  $Hom(D,D').$
If $(\phi_1,\phi_2)$ is morphism from $D$ to $D'$ and $(\psi_1, \psi_2)$ is a morphism from $D'$ to $D''$, then  one can easily verify that $(\psi_1\circ \phi_1, \psi_2\circ \phi_2)$ is a morphism from $D$ to $D''.$ For $(\phi_1,\phi_2)\in Hom(D,D')$  and $(\psi_1,\psi_2)\in Hom(D',D'')$, we define
\begin{equation}
    (\psi_1,\psi_2)\circ (\phi_1,\phi_2)= (\psi_1\circ \phi_1,\psi_2\circ \phi_2).
\end{equation}

For every crossed homomorphisms $D$ from $\g$ to $\h$, $(I_{\g},I_{\h})$ is a morphism from $D$ to itself. For any morphism $(\phi_1,\phi_2)$ from $D$ to $D'$ we have $(I_{\g},I_{\h}) \circ (\phi_1,\phi_2)= (\phi_1,\phi_2)\circ (I_{\g},I_{\h})=(\phi_1,\phi_2)$.  Hence we have following theorem: 
\begin{thm}
   The collection  of all crossed homomorphisms from $\g$ to $\h$  forms a category with Hom set given by Definition \ref{RAD01}. 
\end{thm}


\section{Characterisation and Cohomology of crossed homomorphisms}\label{RAS5}
We define a  $\mathbb{N}\times \mathbb{Z}_2$-graded vector space $C^*(\g,\h)$  by  $C^*(\g,\h)=\oplus_{n=1}^\infty Hom(\wedge^n \g,\h).$  We have $Hom(\wedge^n \g,\h)=Hom_0(\wedge^n \g,\h)\oplus Hom_1(\wedge^n \g,\h),$ where $Hom_i(\wedge^n \g,\h)$ is the 
subspace of $Hom(\wedge^n \g,\h)$ consisting of elements of degree $i,$ $i\in \mathbb{Z}_2.$ We define a bilinear map  $\lb.,.\rb: C^*(\g,\h)\times C^*(\g,\h)\to C^*(\g,\h)$ of degree $(0,0)$ by $$ \lb f_1,f_2 \rb=(-1)^{m-1}[[\mu,f_1],f_2],\;\;
\;\;\forall f_1\in Hom(\wedge^m \g,\h), f_2\in Hom(\wedge^n \g,\h).$$ We have
\begin{align*}
    \lb f_1,f_2 \rb (x_1,\ldots, x_{m+n}) 
   =&(-1)^{m-1}[[\mu,f_1],f_2] (x_1, \ldots, x_{m+n}))\\
    =& ([\mu,f_1] \circ f_2)(x_1, \ldots, x_{m+n}))\\
     =& \sum_{\sigma\in S_{m,n}}\epsilon(\sigma, X) (-1)^{\{s\sum_{i=1}^mx_{\sigma(i)}\}}\\
     &[\mu,f_1] (x_{\sigma(1)}, \ldots, x_{\sigma(m))}, f_2(x_{\sigma (m+1)}, \ldots,  x_{\sigma(m+n)})),\\
      =& \sum_{\sigma\in S_{m,n}}\epsilon(\sigma, X) (-1)^{\{rs+m n-n+1+s\sum_{i=1}^m x_{\sigma(i)}+r\sum_{i=m+1}^{m+n}x_{\sigma(i)}\}}\\
      & [ f_1(x_{\sigma(1)}, \ldots, x_{\sigma(m)}), f_2(x_{\sigma (m+1)}, \ldots,  x_{\sigma(m+n)})], 
\end{align*}
$\forall$ $f_1\in Hom_r(\wedge^m \g,\h),$ $f_2\in Hom_s(\wedge^n \g,\h)$, $X=(x_1,\ldots, x_{m+n})\in \wedge^{m+n} \g$.  Also, we define a linear map  $ \partial_{\pi+\rho}:C^*(\g,\h)\to C^*(\g,\h)$  of degree $(1,0)$ by $\partial_{\pi+\rho}f=[\pi+\rho, f],$ for each $f\in Hom(\wedge^n \g,\h). $ Next theorem gives a characterisation of crossed homomorphisms between  Lie superalgebras  as Maurer-Cartan elements of $\mathbb{N}\times \mathbb{Z}_2$-graded  Lie algebra.

\begin{thm}\label{RA222}
    $(C^*(\g, \h), \lb.,.\rb, \partial_{\pi+\rho})$ is a differential $\mathbb{N}\times \mathbb{Z}_2$-graded  Lie algebra.  A linear map $D: \g \to \h$ of degree $0$ is a crossed homomorphism between Lie superalgebras $\g$ and $\h$ with respect to $\rho$ if and only if  $D$  is a Maurer-Cartan  element of $(C^*(\g,\h), \lb.,.\rb, \partial_{\pi+\rho})$.
\end{thm}
\begin{proof}
We have, for any $X=(X_1,\ldots, X_{m+n})\in \wedge^{m+n} \g$ and for homogeneous $f_1\in Hom(\wedge^m \g,\h)$, $f_2\in Hom(\wedge^n \g,\h)$, $0=[\mu,[f_1,f_2]]X=[[\mu, f_1], f_2]X+(-1)^{m-1}[f_1,[\mu, f_2]]X=[[\mu, f_1], f_2]X+(-1)^{m+mn-n+f_1f_2}[[\mu, f_2], f_1]X$. Hence $ [[\mu, f_1], f_2]+(-1)^{m+mn-n+f_1f_2}[[\mu, f_2], f_1]=0$ on $C^*(\g,\h).$  This implies that $ \lb f_1,f_2 \rb =-(-1)^{mn+f_1f_2}\lb f_2,f_1 \rb.$
\begin{align*}
    \lb f_1,\lb f_2,f_3\rb \rb &=(-1)^{l+m}[[\mu,f_1],[[\mu,f_2],f_3]]\\
    &=(-1)^{l+m}[[[\mu,f_1],[\mu,f_2]],f_3]+(-1)^{l+m+lm+f_1f_2}[[\mu,f_2], [[\mu,f_1],f_3]]\\
    &=(-1)^{m}[ [\mu,[[\mu,f_1],f_2] ],f_3]+(-1)^{l+m+lm+f_1f_2}[[\mu,f_2], [[\mu,f_1],f_3]]\\
    &=\lb \lb f_1,f_2 \rb ,f_3 \rb +(-1)^{lm+f_1f_2} \lb f_2, \lb f_1,f_3 \rb \rb,
\end{align*}
for every homogeneous $f_1\in Hom(\wedge^l \g,\h)$, $f_2\in Hom(\wedge^m \g,\h)$ , $f_3\in Hom(\wedge^n \g,\h)$.
For any $f\in Hom(\wedge^m \g,\h)$  and $X \in \wedge^{m+2}\g,$ we have 
\begin{align}
 0&=[[\pi + \rho + \mu,\pi + \rho + \mu],f]X \nonumber \\
  &=[[\pi + \rho, \pi + \rho]+[\pi + \rho ,\mu]+[\mu,\pi + \rho]+ [\mu,\mu],f]X \nonumber \\
  &=[[\pi + \rho, \pi + \rho]+[\pi + \rho ,\mu]+(-1)^{1+1+0}[\pi + \rho,\mu] + 0 ,f]X \nonumber \\
  &=[[\pi + \rho, \pi + \rho]+2[\pi + \rho ,\mu],f]X \nonumber \\
  &=[[\pi + \rho, \pi + \rho],f]X+ 0 +2[[\rho,\mu],f]X \nonumber \\
  &=[[\pi + \rho, \pi + \rho],f]X + 2\{[\rho,\mu]\circ f -(-1)^{2(m-1)}f \circ[\rho,\mu]\}X \nonumber \\
  &=[[\pi + \rho, \pi + \rho],f]X + 2\{(\rho \circ \mu -(-1)^{1+0}\mu \circ \rho)\circ f -0 \}X \nonumber \\
  &=[[\pi+\rho,\pi + \rho],f]X, \nonumber
\end{align}
Hence   $[\pi+\rho, [\pi+\rho, f]] =[[\pi+\rho, \pi+\rho], f]-[\pi+\rho, [\pi+\rho, f]]=-[\pi+\rho, [\pi+\rho, f]].$ This implies that $[\pi+\rho, [\pi+\rho, f]]=0.$ Thus  $\partial_{\pi+\rho}^2f=[\pi+\rho, [\pi+\rho, f]] =0.$\\
For any $f_1\in Hom_r(\wedge^m \g,\h)$ and $f_2\in Hom_s(\wedge^n \g,\h) $ we have 
\begin{align}\label{RA223}
    &\partial_{\pi+\rho} \lb f_1,f_2 \rb =[\pi+\rho, \lb f_1,f_2 \rb ]\nonumber\\
    &=(-1)^{m-1}[\pi+\rho, [[\mu,f_1],f_2] ]\nonumber\\
    &=(-1)^{m-1}[[\pi+\rho, [\mu,f_1]],f_2] +(-1)^{2m-1}[ [\mu,f_1], [\pi+\rho,f_2]]\nonumber\\
    &= (-1)^{m-1}\{[[[\pi+\rho, \mu],f_1],f_2] -[[\mu, [\pi+\rho,f_1] ],f_2]+(-1)^m[ [\mu,f_1], [\pi+\rho,f_2]]\}.
\end{align}
For any $(X_1,\ldots, X_{m+n+1})\in \wedge^{m+n+1} g$, by direct computation we have 
\begin{align}\label{RA224}
    [[[\pi+\rho, \mu],f_1],f_2](X_1,\ldots, X_{m+n+1})&=((\rho\circ \mu+\mu\circ \rho )\circ f_1)\circ f_2(X_1,\ldots, X_{m+n+1})\nonumber\\
    &=0.
\end{align}
From \ref{RA223} and \ref{RA224}, we conclude that $\partial_{\pi+\rho} \lb f_1,f_2 \rb=\lb \partial_{\pi+\rho}f_1,f_2 \rb +(-1)^m \lb f_1, \partial_{\pi+\rho}f_2 \rb .$
Thus $\partial_{\pi+\rho}$ is a derivation of degree $(1,0).$
Let  $D \in \Hom(\wedge^1 \g,\h)$, $\pi \in \Hom(\wedge^2 \g,\g)$, $\mu \in \Hom(\wedge^2 \h,\h)$, $\rho \in \Hom(\wedge^1 \g \otimes \wedge^1 \h,\h)$. Then 
$\partial_{\pi + \rho}D + \frac{1}{2}\lb D, D \rb = [\pi + \rho, D] +\frac{1}{2} \lb D,D \rb = [\pi, D] + [\rho, D] +\frac{1}{2} \lb D,D \rb $.\\
For any $X=(X_1,X_2) \in \wedge^2 \g $, we have 
$$[\pi,D]X = -(-1)^{\pi D } D (\pi (X_1,X_2)) = -D(\pi(X_1,X_2))$$ 
$$[\rho,D]X =  \rho(X_1,D(X_2)) + (-1)^{X_1X_2  +1} \rho (X_2,D (X_1)), \;\lb D,D \rb X = 2 \mu(D(X_1),D(X_2)).$$
Thus \begin{align*}
    &(d_{\pi + \rho}D + \frac{1}{2}\lb D, D \rb)X \\&= -D(\pi(X_1,X_2)) + \rho(X_1,D(X_2)) - (-1)^{X_1 X_2}\rho(X_2,D(X_1)) + \mu(D(X_1),D(X_2)).
\end{align*} 
Hence, $D$ is a crossed homomorphism between Lie superalgebras $\g$ and $\h$ with respect to $\rho$ if and only if  $D$  is a Maurer-Cartan element of $(C^*(\g,\h), \lb.,.\rb, \partial_{\pi+\rho})$.

\end{proof}
Define $d_D:C^*(\g,\h)\to C^*(\g,\h)$  by $d_Df=\partial_{\pi+\rho} f+\lb D,f \rb ,$  for every $f\in Hom(\wedge^n \g,\h).$  
\begin{thm}
   $(C^*(\g,\h), \lb .,.\rb, d_D)$ is a differential $\mathbb{N}\times \mathbb{Z}_2$-graded  Lie algebra.  
\end{thm}
\begin{proof}
From Theorem \ref{RA222}, $(C^*(\g, \h), \lb.,.\rb)$ is a $\mathbb{N}\times \mathbb{Z}_2$-graded Lie algebra. We just need to show that $d_D$ is a graded derivation and differential of degree $(1,0)$. Since sum of two derivations of degree  $(1,0)$   is a derivation of degree  $(1,0)$, by using Theorem \ref{RA222}, $d_D$ is a derivation of degree $(1,0)$. For $f\in Hom(\wedge^n \g,\h),$  we have $d_D^2f=d_D(\partial_{\pi+\rho} f+ \lb D,f \rb )= \lb D,\partial_{\pi+\rho} f \rb + \lb D, \lb D,f \rb \rb +\partial_{\pi+\rho} \lb D ,f\rb = \lb D,\partial_{\pi+\rho} f \rb +\frac{1}{2} \lb \lb D,D \rb ,f \rb + \lb \partial_{\pi+\rho} D,f \rb -  \lb D,\partial_{\pi+\rho} f \rb = \lb \partial_{\pi+\rho}D+\frac{1}{2} \lb D,D \rb ,f \rb =0.$
\end{proof}
Define $C^n(D)=Hom(\wedge^n\g,\h)$, $\forall n \in \mathbb{N}$.  Write $C^*(D)=\oplus_{n\in \mathbb{N}}C^n(D)$. Then $(C^*(D), d_D)$ is a cochain complex. We call the cohomology of $(C^*(D), d_D)$  as the cohomology of crossed homomorphism $D:\g\to \h$ and denote it by $H^*(D).$  Thus $H^*(D)=\oplus_{n\in \mathbb{N}} H^n(D)$, where  $H^n(D)$ is the $n$-th cohomology of the cochain complex $(C^*(D), d_D)$.

\section{Formal One Parameter Deformation of crossed homomorphisms}\label{RAS6}

    Let $(\g,\pi)$ and $(\h,\mu)$  be  any two Lie superalgebras.  Let $(\g,\h,\rho)$  be a LieSupAct triple and $D:\g\to\h$ be a crossed homomorphism. We write $\g[[t]]=\{\sum_{i=0}^\infty x_it^i: x_i\in \g,\; \forall \; i\}$, $\h[[t]]=\{\sum_{i=0}^\infty y_it^i: x_i\in \h,\; \forall \; i\}$. We write   $D_t=\sum_{i=0}^\infty D_it^i$, where  $D_i:\g\to \h$ are linear maps of degree $0$ for every $i\ge 0$ and  $D_0=D.$  
    \begin{rem}
    \begin{enumerate}
        \item $\pi$ extends to $\tilde{\pi}:\g[[t]]\otimes \g[[t]]\to \g[[t]]$ given by $$\tilde{\pi}(\sum_{i=0}^\infty x_it^i,\sum_{j=1}^\infty y_jt^j)=\sum_{l=1}^\infty\sum_{i+j=l}\pi(x_i,y_j)t^l$$ and $(\g[[t]], \hat{\pi})$ is a Lie superalgebra. Similarly $\mu$ extends to $\tilde{\mu}:\h[[t]]\otimes \h[[t]]\to \h[[t]]$ and  $(\h[[t]], \hat{\mu})$ is a Lie superalgebra.
        \item $\rho $ induces a Lie superalgebra homomorphism $\tilde{\rho}:\g\to Der(\h[[t]],\h[[t]])$ given by $\tilde{\rho}(x)(\sum_{i=1}^\infty y_jt^j)=\sum_{i=1}^\infty\rho(x) (y_j)t^j)$.
        \item For simplicity we denote  $\tilde{\pi}$, $\hat{\mu}$ and $\tilde{\rho}$ by $\pi$ and $\mu$, $\rho$ respectively and recognize  them from context.
    \end{enumerate}
\end{rem}
\begin{defn}\label{RA225}
A formal one parameter deformation of $D$ is a linear map $D_t: \g[[t]]\to \h[[t]]$ such that following condition is  satisfied:
\begin{equation}\label{RA286}
   D_t[x,y]=\rho(x)D_t(y)- (-1)^{\alpha\beta}\rho(y)D_t(x)+[D_t(x), D_t(y)]
\end{equation}
for every   $x\in g_\alpha,$ $ y\in g_{\beta}$. If $D_1\ne 0,$ then it is called infinitesimal of the deformation $D_t.$

\end{defn}
Equation \ref{RA286} is equivalent to following:
\begin{equation}\label{RA227}
   D_n[x,y]=\rho(x)D_n(y)- (-1)^{\alpha\beta}\rho(y)D_n(x)+\sum_{i+j=n}[D_i(x), D_j(y)], \forall n\ge 0.
\end{equation}
In particular, for $n=1$ we have 
\begin{align}\label{RA228}
    D_1[x,y] &=\rho(x)D_1(y)- (-1)^{\alpha\beta}\rho(y)D_1(x)+[D_1(x), D(y)]+[D(x), D_1(y)]\nonumber\\
    &= D_1[x,y]+\partial_{\pi+\rho}D_1(x,y)+ [[D,D_1]](x,y) \nonumber\\
    &=D_1[x,y]+d_DD_1(x,y)
\end{align}
 Using expression \ref{RA228}, we conclude that  $d_DD_1=0,$ that is, $D_1$ is a cocycle for the cochain complex $(C^*(D),d_D)$.  Hence we have following lemma:
\begin{lem}
    Infinitesimal of a deformation is a $1$-cocycle.
\end{lem}
 $\pi$ and $\mu$ extend naturally to make $\g[[t]]/(t^2) $ and  $\h[[t]]/(t^2)$ Lie superalgebras, respectively.  
\begin{defn}\label{RA370} 
A Linear (or degree one or infinitesimal) one parameter deformation of $D$ is a linear map $D_t:\g[[t]]/(t^2) \to \h[[t]]/(t^2)$ such that following condition is  satisfied:
\begin{equation}\label{RA226}
   D_t[x,y]=\rho(x)D_t(y)- (-1)^{\alpha\beta}\rho(y)D_t(x)+[D_t(x), D_t(y)],
\end{equation}
for every   $x\in g_\alpha,$ $ y\in g_{\beta}$, where $D_t=D+D_1t,$  and $D_1:\g\to \h$ is a linear map of degree $0$. 
\end{defn}
 From Equation \ref{RA227}, we conclude following theorem:

 \begin{thm}
     $D+D_1t:\g[[t]]/(t^2)\to \h[[t]]/(t^2)$  is a linear one parameter deformation if and only if $D_1$ is a $1$-cocycle in $(C^*(D),d_D)$.
  \end{thm}

\noindent\textbf{Acknowledgements:} Second author is supported by DST-INSPIRE fellowship by government of India.

\bibliographystyle{alpha}
\bibliography{raj.bib}

\begin{thebibliography}{10}
\expandafter\ifx\csname url\endcsname\relax
  \def\url#1{\texttt{#1}}\fi
\expandafter\ifx\csname urlprefix\endcsname\relax\def\urlprefix{URL }\fi
\expandafter\ifx\csname href\endcsname\relax
  \def\href#1#2{#2} \def\path#1{#1}\fi

\bibitem{MR0438925}
L.~Corwin, Y.~Ne'eman, S.~Sternberg, \href{https://doi.org/10.1103/RevModPhys.47.573}{Graded {L}ie algebras in mathematics and physics ({B}ose-{F}ermi symmetry)}, Rev. Modern Phys. 47 (1975) 573--603.
\newblock \href {https://doi.org/10.1103/RevModPhys.47.573} {\path{doi:10.1103/RevModPhys.47.573}}.
\newline\urlprefix\url{https://doi.org/10.1103/RevModPhys.47.573}

\bibitem{MR486011}
V.~G. Kac, \href{https://doi.org/10.1016/0001-8708(77)90017-2}{Lie superalgebras}, Advances in Math. 26~(1) (1977) 8--96.
\newblock \href {https://doi.org/10.1016/0001-8708(77)90017-2} {\path{doi:10.1016/0001-8708(77)90017-2}}.
\newline\urlprefix\url{https://doi.org/10.1016/0001-8708(77)90017-2}

\bibitem{MR0422372}
D.~A. Le\u{\i}tes, Cohomology of {L}ie superalgebras, Funkcional. Anal. i Prilo\v{z}en. 9~(4) (1975) 75--76.

\bibitem{MR171807}
M.~Gerstenhaber, \href{https://doi.org/10.2307/1970484}{On the deformation of rings and algebras}, Ann. of Math. (2) 79 (1964) 59--103.
\newblock \href {https://doi.org/10.2307/1970484} {\path{doi:10.2307/1970484}}.
\newline\urlprefix\url{https://doi.org/10.2307/1970484}

\bibitem{MR0207793}
M.~Gerstenhaber, \href{https://doi.org/10.2307/1970528}{On the deformation of rings and algebras. {II}}, Ann. of Math. 84 (1966) 1--19.
\newblock \href {https://doi.org/10.2307/1970528} {\path{doi:10.2307/1970528}}.
\newline\urlprefix\url{https://doi.org/10.2307/1970528}

\bibitem{MR240167}
M.~Gerstenhaber, \href{https://doi.org/10.2307/1970553}{On the deformation of rings and algebras. {III}}, Ann. of Math. (2) 88 (1968) 1--34.
\newblock \href {https://doi.org/10.2307/1970553} {\path{doi:10.2307/1970553}}.
\newline\urlprefix\url{https://doi.org/10.2307/1970553}

\bibitem{MR389978}
M.~Gerstenhaber, \href{https://doi.org/10.2307/1970900}{On the deformation of rings and algebras. {IV}}, Ann. of Math. (2) 99 (1974) 257--276.
\newblock \href {https://doi.org/10.2307/1970900} {\path{doi:10.2307/1970900}}.
\newline\urlprefix\url{https://doi.org/10.2307/1970900}

\bibitem{MR704600}
M.~Gerstenhaber, S.~D. Schack, \href{https://doi.org/10.2307/1999369}{On the deformation of algebra morphisms and diagrams}, Trans. Amer. Math. Soc. 279~(1) (1983) 1--50.
\newblock \href {https://doi.org/10.2307/1999369} {\path{doi:10.2307/1999369}}.
\newline\urlprefix\url{https://doi.org/10.2307/1999369}

\bibitem{MR871615}
B.~Binegar, \href{https://doi.org/10.1007/BF00402663}{Cohomology and deformations of {L}ie superalgebras}, Lett. Math. Phys. 12~(4) (1986) 301--308.
\newblock \href {https://doi.org/10.1007/BF00402663} {\path{doi:10.1007/BF00402663}}.
\newline\urlprefix\url{https://doi.org/10.1007/BF00402663}

\bibitem{MR0214636}
A.~Nijenhuis, R.~W. Richardson, Jr., \href{https://doi.org/10.1512/iumj.1968.17.17005}{Deformations of {L}ie algebra structures}, J. Math. Mech. 17 (1967) 89--105.
\newblock \href {https://doi.org/10.1512/iumj.1968.17.17005} {\path{doi:10.1512/iumj.1968.17.17005}}.
\newline\urlprefix\url{https://doi.org/10.1512/iumj.1968.17.17005}

\bibitem{MR1028197}
H.~Benamor, G.~Pinczon, \href{https://doi.org/10.1007/BF00405262}{The graded lie algebra structure of lie superalgebra deformation theory}, Lett. Math. Phys. 18~(4) (1989) 307--313.
\newblock \href {https://doi.org/10.1007/BF00405262} {\path{doi:10.1007/BF00405262}}.
\newline\urlprefix\url{https://doi.org/10.1007/BF00405262}

\bibitem{Whitehead}
J.~H.~C. Whitehead, Combinatorial homotopy ii, Bull. Amer. Math. Soc. 55 (1949) 453--496.

\bibitem{Serre}
J.~P. Serre, Galois Cohomology. Springer, Springer, 1997.

\bibitem{Lue}
A.~Lue, Crossed homomorphisms of lie algebras, Proc. Cambridge Philos. Soc. 62 (1966) 577--581.

\bibitem{Yunhe1}
J.~Jiang, Y.~Sheng, Deformations, cohomologies and integrations of relative difference lie algebras, Journal of Algebra 614 (2023) 535--563.

\bibitem{Apurba}
A.~Das, {Cohomology and deformations of crossed homomorphisms}, Bulletin of the Belgian Mathematical Society - Simon Stevin 28~(3) (2022) 381 -- 397.

\bibitem{Wang}
Y.~Li, D.~Wang, Cohomology and deformation theory of crossed homomorphisms of leibniz algebras, Journal of Algebra and Its Applications 0~(0) (0) 2550195.

\bibitem{MR195995}
A.~Nijenhuis, R.~W. Richardson, Jr., \href{https://doi.org/10.1090/S0002-9904-1966-11401-5}{Cohomology and deformations in graded lie algebras}, Bull. Amer. Math. Soc. 72 (1966) 1--29.
\newblock \href {https://doi.org/10.1090/S0002-9904-1966-11401-5} {\path{doi:10.1090/S0002-9904-1966-11401-5}}.
\newline\urlprefix\url{https://doi.org/10.1090/S0002-9904-1966-11401-5}

\bibitem{GARCIAMARTINEZ2015464}
M.~L. Xabier García-Martínez, Emzar~Khmaladze, Non-abelian tensor product and homology of lie superalgebras, Journal of Algebra 440 (2015) 464--488.

\end{thebibliography}

\end{document}